\documentclass[a4paper]{amsart}
\usepackage[latin9]{inputenc}
\usepackage{a4}
\usepackage{amsfonts}
\usepackage{amssymb}
\usepackage{amsmath}
\usepackage{amsthm}
\usepackage{changepage}
\usepackage{nomencl}
\usepackage{geometry}
\usepackage{tikz}
\usepackage{verbatim}
\usetikzlibrary{matrix}
\usepackage{hyperref}
\usepackage[toc,page]{appendix}
\usepackage{mathrsfs}
\usepackage[square, numbers, comma]{natbib}

\usepackage{enumerate}
\usepackage{amsbsy}

\usepackage[all]{xy}
\usepackage{pb-diagram,pb-xy}
   \bibpunct{[}{]}{,}{n}{}{,}
\input cyracc.def

\begin{document}
\providecommand{\keywords}[1]{\textbf{\textit{Keywords: }} #1}
\newtheorem{theorem}{Theorem}[section]
\newtheorem{lemma}[theorem]{Lemma}
\newtheorem{proposition}[theorem]{Proposition}
\newtheorem{corollary}[theorem]{Corollary}
\theoremstyle{definition}
\newtheorem{definition}{Definition}[section]
\theoremstyle{remark}
\newtheorem{remark}{Remark}
\newtheorem{conjecture}{Conjecture}
\newtheorem{question}{Question}

\def\p{\mathfrak{p}}
\def\q{\mathfrak{q}}
\def\s{\mathfrak{S}}
\def\Gal{\mathrm{Gal}}
\def\Ker{\mathrm{Ker}}
\def\Coker{\mathrm{Coker}}
\newcommand{\cc}{{\mathbb{C}}}   
\newcommand{\ff}{{\mathbb{F}}}  
\newcommand{\nn}{{\mathbb{N}}}   
\newcommand{\qq}{{\mathbb{Q}}}  
\newcommand{\rr}{{\mathbb{R}}}   
\newcommand{\zz}{{\mathbb{Z}}}

\title{Quadratic number fields with unramified $SL_2(5)$-extensions}
\author{Joachim K\"onig}
\address{Korea National University of Education, Department of Mathematics Education, 28173 Cheongju, South Korea}
\begin{abstract} Continuing the line of thought of an earlier work \cite{Koe21}, we provide the first infinite family of quadratic number fields with everywhere unramified Galois extensions of Galois group $SL_2(5)$, the (unique) smallest nonsolvable group for which this problem was previously open. Our approach also improves upon \cite{Koe21} by yielding the first infinite family of {\it real}-quadratic fields possessing an unramified Galois extension whose Galois group is perfect and not generated by involutions. Our result also amounts to a new existence result on quintic number fields with squarefree discriminant and additional local conditions.
 \end{abstract}
\maketitle
\section{Introduction and main result}

%{\bf Motivation: Simple groups are generated by involutions and thus in principle accessible directly, via Abhyankar's lemma. Next step: Quasi-simple groups. These are generated by involutions unless the simple component has generalized quaternion $2$-Sylow group. This list is known by Gorenstein,Walter 1965 and is just $PSL_2(q)$ with $q\ge 5$ odd, or $A_7$. Among all quasi-simple groups $2.S$, covering groups of these stand out as ``difficult" since not directly accessible via Abhyankar's lemma.}

This work continues the investigations of \cite{Koe21}, motivated by the following folklore conjecture in inverse Galois theory.
\begin{conjecture}
\label{conj:1}
For any finite group $G$, there exist infinitely many quadratic number fields $K$ such that $K$ possesses a Galois extension with Galois group $G$ unramified at all (finite and infinite) primes of $K$. 
\end{conjecture}

Stronger versions of Conjecture \ref{conj:1} also take into account the asymptotic distribution of the number of unramified extensions (when the quadratic fields are counted by discriminant), generalizing the Cohen-Lenstra heuristics; see in particular \cite{Wood}.
Conjecture \ref{conj:1} is known to be true only in some special cases, including the case of cyclic groups (where Conjecture \ref{conj:1}  relates directly to the ``class number divisibility problem" of class field theory, cf.\ \cite{AC55}), symmetric and alternating groups (e.g., \cite{Yam70} (which does not consider ramification at infinity) or \cite{Ked} (which does)), and a few other almost simple finite groups (e.g., \cite{KNS}). Groups such as $G=A_n$, and more generally, perfect groups $G$ (that is, groups with trivial abelianization), are in some way on the opposite side of abelian groups, and investigating Conjecture \ref{conj:1} for them may be seen as complementary to the class number divisibility problem. While the above nonsolvable examples share the property that they are generated by involutions, \cite{Koe21} obtained a solution for the first cases of perfect groups not generated by involutions (a property making certain direct approaches via, e.g., Abhyankar's lemma impossible), including the group $SL_2(7)$. %The common idea for the nonsolvable cases is to carefully specialize suitable function field extensions over $\mathbb{Q}$ (or equivalently, multiparameter polynomials over $\mathbb{Q}$) with prescribed Galois group to obtain extensions of $\mathbb{Q}$ with a desired ramification behavior. 

In this paper, we solve the smallest open non-solvable case of the conjecture, namely:

%Note: Smallest non-simple perfect groups: SL_2(5), SL_2(7), SL_(2,9)
%Among simple non-alternating group, only extra composition factor is PSL_2(8)
%Next one would be 2^4.A_5
\begin{theorem}
\label{thm:main}
There exist infinitely many imaginary-quadratic and infinitely many real-quadratic number fields possessing an everywhere (i.e., including the archimedean primes) unramified Galois extension with Galois group $SL_2(5)$.%
%\marginpar{Full covering group $6.A_7$ should also be fine, since in proof, all inertia groups are generated by involutions!} 
%
%Maybe can see this theoretically at least for q=7, from tuple (2,6,7)?
\end{theorem}

We will see that the proof of this result about $G=SL_2(5)$ is significantly more intricate than the one of the (known) analogous result for $A_5 (\cong PSL_2(5))$; progressing from one to the other corresponds to a central embedding problem with kernel $C_2$ and certain extra conditions; those conditions however make the solution rather nontrivial. On a similar note, while $A_5$ has been known to be a Galois group over $\mathbb{Q}$ since the 19th century, $SL_2(5)$ was first shown to be realizable by Sonn \cite{Sonn} in 1980.

Note also that while the proof of Theorem \ref{thm:main} shares some common theoretical ideas with the previously obtained results in \cite{Koe21} for the two perfect groups $G=SL_2(7)$ and $G=2.A_7$, it requires a significant amount of extra computational effort compared to these.
On the other hand, the assertion on real-quadratic number fields in Theorem \ref{thm:main} goes beyond what was obtained for the groups considered in \cite{Koe21}; indeed, the construction there automatically lead only to imaginary-quadratic fields.  We will elaborate in the beginning of Section \ref{sec:proof} on the additional difficulty of the case $G=SL_2(5)$ compared to the previously solved cases. 
 In Corollary \ref{cor:q5}, we give a fully explicit family of quadratic number fields fulfilling the assertion of Theorem \ref{thm:main}. 
 This also yields a lower bound on the asymptotic distribution of such quadratic fields, although obviously one far away from the actual expectation of the generalized Cohen-Lenstra heuristics. Magma code for the verification of several claims about computations is contained in the ancillary file \texttt{sl25\_unram\_anc.txt}.
 
 Finally, the reader may compare works such as \cite{Maire} and \cite{Brink}, with which our result shares the common idea of ``building on top" of an initial nonsolvable unramified extension (although with the notable difference that those papers are not focussed on realizing an extension with a particular Galois group, and instead aim at the construction of {\it infinite} towers).

\section{Prerequisites}
We recall some key results on specializations of function field extensions and on central embedding problems. This section is largely identical with the analogous section in  \cite{Koe21}, and may be skipped by a reader familiar with the techniques.

\subsection{Local behavior of specializations of function field extensions}
For a finite Galois extension $N/K(t)$ of function fields and any value $t_0\in \mathbb{P}^1(K)$, we denote by $N_{t_0}/K$ the specialization of $N/K(t)$ at $t\mapsto t_0$, i.e., the residue extension at any place extending $t\mapsto t_0$ in $N$. Recall that $N/K(t)$ is called $K$-regular if $K$ is algebraically closed in $N$. 
The following well-known theorem, cf.\ \cite[Prop.\ 4.2]{Beck}, ties the ramification in specializations of a function field extension to the inertia groups at branch points of that function field extension. 

%
% Check whether really necessary for this paper!!
%
\begin{theorem}
\label{thm:beck}
Let $K$ be a number field and $N/K(t)$ a finite $K$-regular Galois extension with Galois group $G$.
Then there exists a finite set $\mathcal{S}_0$ of primes, depending only on $N/K(t)$, such that the following holds for every prime $\mathfrak{p}$ of $K$ not in $\mathcal{S}_0$:\\
If $t_0\in K$ is not a branch point of $N/K(t)$ then the following condition is necessary for $\mathfrak{p}$ to be ramified in the specialization $N_{t_0}/K$:
 $$e_i:=I_{\mathfrak{p}}(t_0, t_i)>0 \text{ for some (automatically unique, up to algebraic conjugates) branch point $t_i$.}$$
 Here $I_{\mathfrak{p}}(t_0,t_i)$ is the intersection multiplicity of $t_0$ and $t_i$ at the prime $\mathfrak{p}$.
Furthermore, the inertia group of a prime extending $\mathfrak{p}$ in the specialization $N_{t_0}/K$ is then conjugate in $G$ to $\langle\tau^{e_i}\rangle$, where $\tau$ is a generator of an inertia subgroup over the branch point $t\mapsto t_i$ of $K(t)$.
\end{theorem}

Regarding the definition of intersection multiplicity $I_{\mathfrak{p}}$ occurring in Theorem \ref{thm:beck}, note that in the special case $K=\mathbb{Q}$ this may be defined conveniently in the following way: Let $f(X)\in \zz[X]$ be the irreducible polynomial of $t_i$ over $\zz$ and $\tilde{f}(X,Y)$ its homogenization (with $\tilde{f}:=Y$ in the special case $t_i=\infty)$. Let $t_0=\frac{a}{b}$ with $a,b\in \zz$ coprime, and let $p$ be a prime number. Then $I_p(t_0,t_i)$ is the multiplicity of $p$ in $\tilde{f}(a,b)$.

The following result (cf.\ \cite[Thm.\ 4.1]{KLN}) clarifies the structure of decomposition groups in specializations.
To state it, we need to introduce a bit more notation. Given a finite Galois extension $N/K(t)$ with group $G$ and a $K$-rational place $t\mapsto a\in \mathbb{P}^1(K)$, denote by $D_a$ and $I_a$ the decomposition and inertia group at the place $t\mapsto a$ respectively. Furthermore, given a prime $p$ of $K$, denote by $D_{a,p}$ the decomposition group at $p$ in the residue extension $N_a/K$. Note that the Galois group of $N_a/K$ is identified with $D_a/I_a$, and hence $D_{a,p}$ is identified with a subgroup of this group.

\begin{theorem}
\label{thm:kln}
%(Theorem on local behaviour of specializations, from \cite{KLN})
Let $K$ be a number field and $N/K(t)$ a finite $K$-regular Galois extension with Galois group $G$.
Let $t\mapsto t_i \in \mathbb{P}^1(\overline{K})$ be a branch point of $N/K(t)$, and let $I_{t_i}$ and $D_{t_i}$ denote the inertia and decomposition group at $t\mapsto t_i$ in $N(t_i)/K(t_i)(t)$.
Then there exists a finite set $\mathcal{S}_1$ of primes of $K$, depending only on $N/K(t)$ and containing the set $\mathcal{S}_0$ from Theorem \ref{thm:beck}, such that for all primes $\mathfrak{p}$ of $K$ not in $\mathcal{S}_1$ and all non-branch points $a\in \mathbb{P}^1(K)$ of $N/K(t)$, the following hold:

\begin{itemize}
\item[a)] Assume that $\mathfrak{p}$ ramifies in $N_a/K$ and $I_{\mathfrak{p}}(a,t_i)>0$.\footnote{The latter is guaranteed to happen for some branch point $t_i$, unique up to algebraic conjugates, by Theorem \ref{thm:beck}.} 
Let $D_{t_i,\mathfrak{p}'}$ denote the decomposition group of $N(t_i)_{t_i}/K(t_i)$ at the (unique) %\footnote{Here as well, uniqueness can be forced up to suitably increasing the exceptional set $\mathcal{S}_1$, if necessary.} 
 prime $\mathfrak{p}'$ of $K(t_i)$ extending $\mathfrak{p} $ such that $I_{\mathfrak{p}'}(a,t_i)>0$. Then the decomposition group $D_{a,\mathfrak{p}}$ at $\mathfrak{p}$ in $N_a/K$ is identified (up to conjugacy in $G$) with a subgroup of $D_{t_i}$, and fulfills $\varphi(D_{a,\mathfrak{p}}) = D_{t_i,\mathfrak{p}'}$, where $\varphi:D_{t_i}\to D_{t_i}/I_{t_i}$ is the canonical epimorphism.
\item[b)] In particular, if additionally $I_{\mathfrak{p}}(a,t_i) = 1$, then $D_{a,\mathfrak{p}} = \varphi^{-1}(D_{t_i,\mathfrak{p}'})$.
\end{itemize}
\end{theorem}

\begin{remark}
\label{rem:explicit}
%\begin{itemize}
%\item[a)]
As noted in \cite[Theorem 2.2]{KN21}, the exceptional set $\mathcal{S}_1$ in Theorem \ref{thm:kln} can be effectively bounded from above as follows. If $O_K$ is the ring of integers of $K$ and $f(t,X)\in O_K[t,X]$ is a polynomial with splitting field $N$, then a prime $p$ of $K$ is not in $\mathcal{S}$ as soon as all of the following are fulfilled: a) $p\not | |G|$; furthermore, if $\Delta\in O_K[t]$ denotes the discriminant of $f$ with respect to $x$, then b) $p$ does not divide the leading coefficient of $\Delta$, and c) the number of distinct roots of $\Delta$ does not decrease upon reduction modulo $p$.
\iffalse
The finite exceptional set $\mathcal{S}_1$ in Theorem \ref{thm:kln} can be made explicit, and this is sometimes relevant when doing calculations with concrete extensions, as we will do in Theorem \ref{thm:q5}. In order to avoid too much technicality, we only refer to \cite[Theorem 2.2]{KN21} for an effective version of Theorem \ref{thm:kln}.
\item[b)] For many applications, it is enough to obtain {\it some} specialization values $a\in \mathbb{P}^1(K)$ yielding a certain prescribed local behavior, rather than controlling the behavior for all specialization values. In this context, the notion of a {\it universally ramified prime} (see \cite{BSS} or \cite{KNS}) becomes useful: let $\mathcal{U}(N/K(t))$ be the set of all primes of $K$ ramifying in all specializations $N_a/K$ where $a\in \mathbb{P}^1(K)$ is a non-branch point of $N/K(t)$. Let $\mathcal{S}$ be any finite set of primes of $K$ disjoint from $\mathcal{U}(N/K(t))$. Then there exists a non-empty and $\mathcal{S}$-adically open set of values $a\in \mathbb{P}^1(K)$ such that $N_a/K$ is unramified at all $p\in \mathcal{S}$ (an easy consequence of Krasner's lemma). In particular, by choosing $\mathcal{S}:=\mathcal{S}_1 \setminus \mathcal{U}(N/K(t))$ with $\mathcal{S}_1$ as in Theorem \ref{thm:kln}, we see that for all such values $a$, the assertion of Theorem \ref{thm:kln} actually holds with $\mathcal{S}_1$ replaced by $\mathcal{U}(N/K(t))$, and the latter set is often much more accessible and can be bounded very effectively by simply comparing some very few specializations.
\end{itemize}
\fi
\end{remark}

\subsection{Embedding problems}
\label{sec:embed_basics}
We recall some basic terminology around embedding problems.   
A {\it finite embedding problem} over a field $K$ is a pair $(\varphi:G_K\to G,\varepsilon:\tilde{G} \to G)$, where $\varphi$ is a (continuous) epimorphism from the absolute Galois group $G_K$ of $K$ onto $G$, 
and $\varepsilon$ is an epimorphism between finite groups $\tilde{G}$ and $G$ fitting in an exact sequence $1\to N\to \tilde{G}\to G\to 1$. The kernel $N=\ker(\varepsilon)$ is called the kernel of the embedding problem. An embedding problem is called {\it central} if $\ker(\varepsilon) \le Z(\tilde{G})$. %, and {\it Frattini} if $\ker(\varepsilon)$ is contained in the Frattini subgroup of $\tilde{G}$.
A (continuous) homomorphism $\psi:G_K\to \tilde{G}$ is called a \textit{solution} to $(\varphi, \varepsilon)$ if the composition  $\varepsilon \circ \psi$ equals $\varphi$. 
In this case, the fixed field of $\ker(\psi)$ is called a {\it solution field} to the embedding problem.
A solution $\psi$ is called a \textit{proper solution} if it is surjective. In this case, the field extension of the solution field over $K$ has full Galois group $\tilde{G}$. 
%Given a finite vector ${\bf t} = (t_1,\dots, t_r)$ of independent transcendentals, an embedding problem $(\varphi:G_K\to G, \varepsilon)$ can be lifted to an embedding problem $(\varphi^\star: G_{K({\bf t})}\to G, \varepsilon)$ over $K({\bf t})$ by identifying the Galois group $G = \Gal(F/K)$ (where $F=\Fix(\ker(\varphi))$) with $\Gal(F({\bf t})/K({\bf t}))$.
%A solution $\psi$ to the latter embedding problem is called {\it regular} if $F$ is algebraically closed in the fixed field of $\ker(\psi)$.

If $K$ is a number field and $\mathfrak{p}$ is a prime of $K$, every embedding problem $(\varphi, \varepsilon)$ induces an associated {\it local embedding problem} $(\varphi_{\mathfrak{p}}, \varepsilon_{\mathfrak{p}})$ defined as follows: $\varphi_{\mathfrak{p}}$ is the restriction of $\varphi$ to $G_{K_{\mathfrak{p}}}$ (well defined up to fixing an embedding of $\overline{K}$ into $\overline{K_{\mathfrak{p}}}$), and $\varepsilon_{\mathfrak{p}}$ is the restriction of $\varepsilon$ to $\varepsilon^{-1}(G({\mathfrak{p}}))$, where $G(\mathfrak{p}) := \varphi_{\mathfrak{p}}(G_{K_{\mathfrak{p}}})$.

We require two well-known results on central embedding problems.
The first  (cf.\ \cite[Chapter IV, Cor.\ 10.2]{MM}) is a local global-principle for the solvability of such embedding problems, which is essentially a consequence of the local-global principle for Brauer embedding problems over number fields (\cite[Chapter IV, Cor.\ 7.8]{MM}).

\begin{proposition}
\label{lem:prime_kernel}
Let $\Gamma = C.G$ be a central extension of $G$ by a cyclic group $C$ of prime order and $\varepsilon:\Gamma\to G$ the canonical projection.
Let $\varphi:G_{\mathbb{Q}}\to G$ be a continuous epimorphism. Then the embedding problem 
 $(\varphi, \varepsilon)$ is solvable if and only if all associated local embedding problems $(\varphi_p, \varepsilon_p)$ are solvable, where $p$ runs through all primes of $\mathbb{Q}$ (including the infinite one). 
\end{proposition}
\begin{remark}
\label{rem:prime_kernel}
If additionally $|C|=2$, the above local-global principle holds already with the set of all primes replaced by ``the set of all primes, with one exception", hence, e.g., with the set of all finite primes. This is due to the product formula for Hasse invariants for Brauer embedding problems.
\end{remark}

Since the local embedding problem at an unramified prime is always solvable, the above lemma gives an efficient method to check for global solvability of a given embedding problem by investigating only the finitely many ramified primes.

The second result (see \cite[Prop.\ 2.1.7]{Serre}) is a tool to control ramification in solutions of central embedding problems, assuming solvability. 
\begin{proposition}
\label{prop:serre}
Let $\Gamma = C.G$ be a central extension of $G$ by a finite abelian group $C$, let $\varepsilon:\Gamma\to G$ be the canonical projection, and let $\varphi: G_{\mathbb{Q}}\to G$ be a continuous epimorphism such that the embedding problem $(\varphi,\varepsilon)$ has a solution. 
For each finite prime $p$, let $\tilde{\varphi_p}: G_{\mathbb{Q}_p}\to \Gamma$ be a solution of the associated local embedding problem $(\varphi_p,\epsilon_p)$, 
 chosen such that all but finitely many $\tilde{\varphi_p}$ are unramified. Then there exists a (not necessarily proper) solution $\tilde{\varphi}: G_{\mathbb{Q}}\to \Gamma$ of $(\varphi,\epsilon)$ such that for all finite primes $p$, the restrictions of $\tilde{\varphi}$ and $\tilde{\varphi_p}$ to the inertia group inside $G_{\mathbb{Q}_p}$ coincide. In particular, $\tilde{\varphi}$ is ramified exactly at those finite primes $p$ for which $\tilde{\varphi_p}$ is ramified.
\end{proposition}

\section{Proof of Theorem \ref{thm:main}} 
\label{sec:proof}

%\subsection{The case $G=SL_2(5)$}
\subsection{Preliminary considerations}

The basic approach in the proof of the case $G=SL_2(5)$ is the same as for the realization of $SL_2(7)$ in \cite{Koe21}: we try to realize $PGL_2(5)\cong S_5$ over $\mathbb{Q}$ in such a way that all nontrivial inertia groups are generated by involutions outside of $PSL_2(5)$ (i.e., in the representation of $PGL_2(5)$ as $S_5$, by transpositions), and such that furthermore no local obstructions to the central embedding problem given by $1\to C_2\to 2.S_5 \to S_5\to 1$ arise. Here, $2.S_5$ denotes the stem extension of $S_5$ in which the transpositions of $S_5$ are split. The centralizer of such an involution in $S_5$ is conjugate to $\langle(1,2), Sym(\{3,4,5\}) \rangle \cong C_2\times S_3$. Since the subgroup $C_6\le C_2\times S_3$ of $S_5$ is split in the central extension, the only local obstruction to the embedding problem which can arise at a tamely ramified prime with inertia group $C_2$ (generated by a transposition) occurs when the decomposition group equals $C_2\times C_2$.

Our main task is therefore to guarantee the existence of tamely ramified $S_5$-extensions of $\mathbb{Q}$ in which all inertia groups at ramified primes are generated by transpositions and such that the corresponding decomposition groups are additionally cyclic (namely, of order $2$ or $6$). Recalling additionally that tame $S_n$-extensions in which all non-trivial inertia groups at finite primes are generated by transpositions are exactly those whose degree-$n$ subextension has squarefree discriminant, we thus have a nice translation of our problem into a combination of two previously-investigated properties.

\begin{lemma}
\label{lem:sqf-loccyc}
Let $F/\mathbb{Q}$ be a degree-$5$ extension with Galois closure $\Omega/\mathbb{Q}$ fulfilling both of the following:
\begin{itemize}
\item[1)] The discriminant $\Delta(F)$ is squarefree.
\item[2)] $\Omega/\mathbb{Q}$ is ``locally cyclic", i.e., all its decomposition groups are cyclic.
\end{itemize}
Then $\Omega/\mathbb{Q}(\sqrt{\Delta(F)})$ embeds into an $SL_2(5)$-extension unramified at all finite primes.
\end{lemma}
\begin{proof}
%\marginpar{Also need to assume that complex conjugation is not a double transposition!!!!}
%%% No, since for kernel C2 may leave out one prime (Hasse invariant)!!! State this as remark after Prop.\ref{lem:prime_kernel}!
Condition 1) translates to saying that $F/\mathbb{Q}$ is tamely ramified with all non-trivial inertia groups generated by transpositions. Since a transitive subgroup of $S_n$ generated by transpositions is necessarily $S_n$ itself, we have $\textrm{Gal}(\Omega/\mathbb{Q})\cong S_5$ and $\textrm{Gal}(\Omega/\mathbb{Q}(\sqrt{\Delta(F)})) \cong A_5$, with $\Omega/\mathbb{Q}(\sqrt{\Delta(F)})$ unramified at all finite primes.  As deduced above, the central embedding problem induced by $1\to C_2\to 2.S_5\to S_5\cong \textrm{Gal}(\Omega/\mathbb{Q})\to 1$ has no local obstruction at any ramified finite prime, and hence no local obstruction at {\it any} finite prime. Due to  Proposition \ref{lem:prime_kernel} and Remark \ref{rem:prime_kernel}, it is therefore solvable (and automatically properly solvable, since the group extension is non-split). Using additionally Proposition \ref{prop:serre} (together with the fact that all the subgroups generated by transpositions split in the extension $2.S_5$), we see that the solution field $L$ may be chosen such that $L/\Omega$, and hence even $L/\mathbb{Q}(\sqrt{\Delta(F)})$, is unramified at all finite primes\end{proof}

Number fields $F$ fulfilling part 1) of Lemma \ref{lem:sqf-loccyc} are indeed well-known for all degrees, see, e.g., \cite{Yam70} or \cite{Ked}; for quintics, Bhargava has obtained much stronger results on the asymptotic distribution of such fields, see in particular \cite[Theorem 1.3]{Bharg}. Also, part 2) was obtained for $S_5$ in \cite{KK21} (however, explicitly {\it not} fulfilling 1) at the same time). In order to get both properties simultaneously, we search for geometric $S_5$-realizations with suitable local behavior at ramified places. In \cite{Koe21}, the analog (with condition 1) suitably adapted) for $G=SL_2(7)$ was obtained by using suitable ``three branch point" extensions with group $PGL_2(7)$. The main technical tool was a ``pullback" construction which, roughly speaking, allowed to discard two of the three branch points, while favorably manipulating the residue extension at the third one.

 However, we face one technical problem: while there are of course some very easy polynomials $f(t,X)$ with Galois group $S_5$ over $\mathbb{Q}(t)$ (such as $f(t,X) = X^5-t(X-1)$), the translation process of \cite{Koe21} seems to result in certain ``badly behaving" primes. E.g., for the polynomial $X^5-t(X-1)$, the residue extension at the only branch point  $t_0$ with inertia group of order $2$ contains the quadratic field $\mathbb{Q}(\sqrt{-2})$. Since $\mathbb{Q}(\sqrt{t_0}) = \mathbb{Q}(\sqrt{5})$, one may ``swallow" this quadratic field via translate $t(s):=-10s^2$. However, for the translated polynomial $f(t(s),X) = X^5+10s^2(X-1)$, the prime $p=5$ will now ramify of ramification index $>2$ in all specializations. The same in fact seems to occur for all three-branch-point extensions with Galois group $S_5$. One therefore needs to dig a bit deeper and search among some 4-branch point extensions with group $S_5$, corresponding to finding suitable members of families $f(a,t,X)$ of $S_5$-polynomials with four branch points (with $a$ to be specialized suitably in the process). Adding extra branch points comes with the downside of having to control more and more residue extensions at branch points simultaneously.  There is, however, also a notable upside: By a well-known argument due to Serre (e.g., Example 10.2 in \cite[Chapter I.10]{MM}), a three branch point Galois extension of $\mathbb{Q}(t)$ with Galois group $G$ cannot possible specialize to totally real extensions, except for a few small dihedral groups $G$, meaning that via a specialization approach using such an extension one cannot hope for {\it real} quadratic fields with everywhere (namely, including the infinite primes) unramified $G$-extensions. Our 4-branch point approach below will, however, succeed. Theorem \ref{thm:q5} below identifies quadratic number fields with an unramified $SL_2(5)$-extension inside an explicitly given infinite family.
 %; part b) goes farther by identifying a subfamily of quadratic fields every single member of which has the above property.

\subsection{A strong version of Theorem \ref{thm:main} and its proof}

\begin{theorem}
\label{thm:q5}
Assume that $w=\frac{w_1}{w_2}\in \mathbb{Q}$ (with coprime integers $w_1,w_2$) fulfills all of the following:
\begin{itemize}
\item[a)] $(-50w_1^2+27w_2^2)(10w_1^2+8w_1w_2+w_2^2)\in \mathbb{Z}$ is equivalent, up to squares, to a squarefree number all of whose prime divisors are congruent to $1$ modulo $3$.
%\item[b)] $w_1$ is odd. \marginpar{This may be dropped??}
\item[b)] %$w\notin [-0.354, -0.254] \cup [-0.238, 4.89]$.
%\marginpar{Intervals based on earlier false parameterization. Recalculate!}
$w\ne 0$, and $w \in [-0.7348,-0.645] \cup [-0.155, 0.7348]$. 
\end{itemize}

Then there exist infinitely many $s\in \mathbb{Q}$ such that the following holds: If $\Delta_s$ denotes the discriminant of the polynomial $$f_{w,s}(X) = X^2(X-1)^3+2s^2(50w^2-27)(10w^2+8w+1)(X-2w^2),$$ 
then the
quadratic number field $\mathbb{Q}(\sqrt{\Delta_s})$ possesses an everywhere (i.e., including the infinite primes) unramified $SL_2(5)$-extension.
Moreover, the root fields of $f_{w,s}$ fulfill the assumptions of Lemma \ref{lem:sqf-loccyc}, and there are infinitely many \underline{distinct} quadratic fields among the fields $\mathbb{Q}(\sqrt{\Delta_s})$.

If furthermore 
\begin{itemize}
\item[b')] %$w\in [-0.254, -0.238] \cup [8.234, 8.262]$,
$|w|\in [0.645, 0.7348]$
\end{itemize} then among those quadratic fields, there are infinitely many real ones.
\iffalse
%code for real roots:
p<w>:=PolynomialRing(Integers());
q<t>:=PolynomialRing(p);
q2<x>:=PolynomialRing(q);
f:=x^3*(x-1)^2-2*t^2*(43*w^2-200*w-50)*(97*w^2+64*w+10)*(9*w^2*x-(w^2-8*w-2));
d:=Discriminant(f);
lc:=LeadingCoefficient(d);
fac:=Factorization(lc);
r:=[rr[1]: rr in Roots(PolynomialRing(RealField())!&*[a[1]: a in fac])];
e:=Discriminant(&*[a[1]: a in Factorization(d)]);
e2:=&*[a[1]: a in Factorization(e)];
for rr in Roots(PolynomialRing(RealField())!e2) do if not rr[1] in r then Append(~r, rr[1]); end if; end for;
r:=Sort(r);r;

test:=[BestApproximation(r[1]-1, 10^6)]; 
for j:=1 to #r-1 do Append(~test,BestApproximation((r[j]+r[j+1])/2,10^6)); end for;
Append(~test, BestApproximation(r[#r]+1, 10^6)); test;

for w in test do
p<t>:=PolynomialRing(Rationals());
q<x>:=PolynomialRing(p);
f:=x^3*(x-1)^2-2*t^2*(43*w^2-200*w-50)*(97*w^2+64*w+10)*(9*w^2*x-(w^2-8*w-2));
d:=Discriminant(f);fac:=Factorization(d);
s:=[rr[1]: rr in Roots(PolynomialRing(RealField())!&*[a[1]: a in fac])];
s2:=[s[1]-1];
for j:=1 to #s-1 do Append(~s2, (s[j]+s[j+1])/2); end for;
Append(~s2, s[#s]+1);
for t0 in s2 do x:=t; #Roots(PolynomialRing(RealField())!(x^3*(x-1)^2-2*t0^2*(43*w^2-200*w-50)*(97*w^2+64*w+10)*(9*w^2*x-(w^2-8*w-2))));
end for;
<w*1.0>; end for;
\fi
\end{theorem}

A few remarks on the above theorem are in order.
\begin{remark}
\begin{itemize}
\item[a)] There do indeed exist values $w\in \mathbb{Q}$ fulfilling all the assumptions (including condition b')) of the theorem, the easiest being $w=-2/3$, for which the expression in a) is equivalent (up to squares) to $43$. In fact, there are infinitely many such $w$; for example, the curve $43Y^2 = (-50W^2+27)(10W^2+8W+1)$ is a genus-$1$ curve with a rational point (e.g., at $W=-2/3$), hence elliptic, and its rank is positive (which may be quickly concluded from Mazur's torsion theorem by simply enumerating more than $16$ points; e.g., there are $20$ such points whose $W$-value has numerator and denominator of absolute value $\le 300$. Direct computation with Magma actually yields that the rank is $2$). But rational points of an elliptic curve of positive rank are well-known (by the ``Poincar\'e-Hurwitz theorem" \cite[Satz 13]{Hurw}) to be locally dense in the real topology, 
whence infinitely many values $w$ fulfilling all the conditions exist in the vicinity of $w_0=-2/3$.
\item[b)] A different, but related question is whether one may even find quadratic fields $\mathbb{Q}(\sqrt{\Delta_s})$ as in the assertion such that any prescribed finite set of primes is unramified in $\mathbb{Q}(\sqrt{\Delta_s})$. This would require finding not just infinitely many admissible values $w$ as above, but infinitely many pairwise coprime squarefreee integers $D>0$ with all prime divisors congruent to $1$ modulo $3$, and such that the curve $DY^2 = (-50W^2+27)(10W^2+8W+1)$ has positive rank (since, indeed, calculation of $\Delta_s$ shows that the prime divisors of $D$ ramify in $\mathbb{Q}(\sqrt{\Delta_s})$). Commonly believed conjectures about distribution of ranks of twists of elliptic curves would suggest that this can be achieved, and one may easily enumerate a great lot of such values $D$, although proving an infinity result might be hard (and, in fact. it is exactly the difficulty of such problems on polynomial values with all prime divisors in some prescribed set which makes direct approaches to Theorem \ref{thm:main} problematic). Nevertheless, even calculation of a few values $w$ (together with the corresponding discriminant $\Delta_s$) already shows that, e.g., for any single prime $p$, there exist such fields $\mathbb{Q}(\sqrt{\Delta_s})$ which are additionally unramified at $p$.
\end{itemize}
\end{remark}

We furthermore note that an infinite set of values $s$ and of the associated $\Delta_s$ as above can be effectively computed for each $w$ (as the proof will show). We content ourselves with the case $w=-2/3$ (and specialization values leading to imaginary quadratic fields), for which the following explicit version may be deduced.
\begin{corollary}
\label{cor:q5}
For all integers $u$ coprime to $30$, 
 the (imaginary) quadratic number field $$\mathbb{Q}(\sqrt{-43\cdot (2^{11}\cdot 3^{10}\cdot 43^6 u^{12} - 43^3\cdot 263\cdot 883\cdot u^6 + 108)})$$ possesses an everywhere unramified $SL_2(5)$-extension.
\end{corollary}

\begin{proof}[Proof of Theorem \ref{thm:q5}]
Begin with the family $f(a,t,X) = X^2(X-1)^3 - t(X-a)$. 
%
%
% Notable property: Pullback by quadratic is still genus-0!!! (namely (3), (2.2), 4 transpositions)
% one possible parameterization:  s;= a/c^2*(u^2 -8/9*((w-1/4)*(w+1/2)/(c*w^2))) / (u*(u^2-1/c)^2);   %%% COMPARISON OF DISCRIMINANT SHOWS THIS IS DOUBTFUL; MAYBE CONSTANT FACTOR WRONG....
% concretely, when w=-1/4, this gives...?  [Is it visible without evaluating that 43 is the only universally ramified prime here?]
% Residue field at s=0 would be ramified only at 2,3 and 43.
% Exceptional set for this extension, as in KN20, would be 2,3,5,43,643,9649.

% For original extension (before pullback) would have been 2,3,5,43,97. Does this mean the two large primes above are not relevant?
% E.g. assume above 643 ramifies (specialization s->s0 meaning t->86 s_0^2 below) then 643 must divide m(86*s_0^2).
% checked: 643 splits in quadr. subextension (= in residue field of the quadruple); and 9649 cannot even ramify with transp. inertia due to being inert in starting ext....
% 
% NOTE: Prime 97 DOES behave ``badly" when ramifying, so this needs to be taken into account when describing explicit sets!!!
%
This yields an $S_5$-extension $E/\mathbb{Q}(a)(t)$ ramified at $t=0$ (with inertia group generator of cycle structure $(3.2)$), at $t\mapsto \infty$ (with inertia group generated by a $4$-cycle) and at the two roots $\lambda_1,\lambda_2$ of an irreducible quadratic polynomial $g_a(t)$ (with inertia group generated by a transposition). One may verify that the residue extensions in $E$ at $t\mapsto \lambda_i$ are $S_3$-extensions (i.e., the decomposition group at each of these branch points is isomorphic to $C_2\times S_3 \le S_5$). 

{\it Step 1: Shrinking of residue extensions via pullback}: 
In view of our observations at the beginning of this section, 
we would rather prefer those decomposition groups to be contained inside $C_2\times C_3$. Following the approach of \cite{Koe21}, we hope to arrive at such a scenario via a suitable quadratic pullback $t=t(s) = c\cdot s^2$ (for a nonzero $c$ to be determined). This will not succeed generically, but can be translated into a very explicit system of equations, solvable via standard Groebner basis methods, which then parameterizes those rational values of $a$ (and the corresponding values of $c$) for which the pullback {\it does} succeed at shrinking the decomposition group. Concretely, the branch points of the ``pulled-back" extension $E(s)/\mathbb{Q}(a)(s)$ are $s=0, \infty$, and the roots of $cs^2=\lambda_i$ ($i=1,2$)\footnote{Note that it is ``generically" sufficient to do all the following for $i=1$, since the same result for $i=2$ follows from algebraic conjugacy of the two branch points. We will verify a while later that this algebraic conjugacy will in fact be preserved for {\it all} rational values of $a$ relevant to us.}.

%\marginpar{This footnote seems safe only for values $w$ which eventually do not make both branch points rational!}
% Discriminant calculation yields necessary condition for branch points to be rational:
% (97w^2+64w+10)(w^2+16w+10) = y^2
% Rank-0 elliptic curve with only 4 rational points: (0,\pm 10) and (-1/2, \pm 9/4). ---> no new exceptions!

 On the other hand, it is easy (using Magma) to find the discriminant $\Delta_i \in \mathbb{Q}(a,\lambda_i)$ of the residue extension at $t\mapsto \lambda_i$ in $E/\mathbb{Q}(a)(t)$; the $2$-part of this residue extension is then of course simply $\mathbb{Q}(a,\lambda_i)(\sqrt{\Delta_i})/\mathbb{Q}(a,\lambda_i)$, and since $\mathbb{Q}(a)(s)/\mathbb{Q}(a)(t)$ is unramified at $t\mapsto \lambda_i$, the $2$-part of the residue extension at $t\mapsto \lambda_i$ in the whole compositum $E(s)/\mathbb{Q}(a)(t)$ is the compositum of quadratic extensions $\mathbb{Q}(a,\lambda_i)(\sqrt{\Delta_i}, \sqrt{c\lambda_i})/\mathbb{Q}(a,\lambda_i)$.
 In particular, if $\mathbb{Q}(a,\lambda_i)(\sqrt{\Delta_i}) = \mathbb{Q}(a, \sqrt{c\lambda_i})$ -- which of course translates to $c\lambda_i\Delta_i$ being a non-zero square in $\mathbb{Q}(a,\lambda_i)$ --  then the residue extensions at the branch points $s\mapsto \pm \sqrt{\lambda_i/c}$ of $E(s)/\mathbb{Q}(a)(s)$ are of odd degree. Since $\mathbb{Q}(a,\lambda_i) = \mathbb{Q}(\sqrt{p(a)})$ for a suitable polynomial $p(a)\in \mathbb{Z}(a)$, we can explicitly write $c\lambda_i\Delta_i = \alpha_1 + \alpha_2\sqrt{p(a)}$ (with $\alpha_i\in \mathbb{Q}(a)$), and equate coefficients (at $1$ and $\sqrt{p(a)}$) in the equation $(\alpha_1+\alpha_2\sqrt{p(a)}) = (u + v\sqrt{p(a)})^2$, yielding a reasonably managable system of equations in the unknowns $a,c,u,v$ (indeed, since $c$ is unique only up to squares, we may even demand $u=1$). It turns out that solutions of this system are parameterized by a (rational) curve of genus $0$, and it is then standard (again, when assisted by a computer) to find a parameter $w$ of this curve. We are therefore able to express $a$ and $c$ as rational functions in $\mathbb{Q}(w)$. In our case, we concretely obtained $a=2w^2$,
%
% Should remove denominator from $a$ via transformation in s? (if so, careful with later concrete value!)
%
 and $c=2\cdot(27-50w^2)(10w^2+8w+1)$, which yields exactly the polynomial $f_{w,s}(X)$ in the assertion of the theorem. 
\iffalse

Possibly more principled argument using explicit residue extension at 0 after suitable pullback (compatible with above):
 t = x^3*(x-1)^2/(9*w^2*x - (w^2-8*w-2))
     
     --:> x->0;  -(w^2-8w-2)t = x^3
           x->1:  (8w^2+8w+2)t = y^2
           
           ---> Q(zeta_3)(( (8w^2+8w+2)^3/(w^2-8w-2)^2 \sqrt[6]{t}))
           
           PB by t=c^3s^6    ---> Q(zeta_3, ((43w^2-200w-50)^3(97w^2+64w+10)^3*(w^2-8w-2)^2/(2w+1)^6)^{1/6}) = Q(zeta_3, c^{1/2}*(w^2-8w-2)^{1/3})
           Suffices that the ``non-square part" under the root is positive with all prime divisors of odd multiplicity being p=1 mod 3.
           Corresponds to quadratic twists of g=1 curve Y^2 = (43w^2-200w-50)(97w^2+64w+10)
           
           How much is known here? Enough to give a result on ``fixed finite set unramified" for the quadratic field?!?

\fi

We may now plug in rational values for $w\in \mathbb{Q}$ in the polynomial $f_w(s,X)$. Apart from finitely many values $w$,\footnote{Concretely, an exception occurs only for values $w$ lying on the boundary of the underlying ``moduli space", i.e., which make two or more branch points of the extension coalesce. Calculation of the discriminant shows quickly that, in fact, this does not happen for any rational $w\ne 0$.} this parameterizes a $\mathbb{Q}$-regular $S_5$-extension, say $F^w/\mathbb{Q}(s)$. We can also verify now that the two non-zero finite branch points of the extension given by $X^2(X-1)^3-t(X-2w^2)=0$ are in fact {\it always} algebraically conjugate (a generically fulfilled assumption which we used in the calculations above): indeed, straightforward calculation of discriminants shows that $w$-values for which both these points become rational would have to correspond to rational points on the curve $Y^2 = (10W^2+8W+1)(10W^2-8W+1)$; 
 this is isomorphic to the elliptic curve $Y^2=X(X-15)(X-24)$, and Magma confirms that this elliptic curve has rank $0$ and no rational points apart from the $2$-torsion - these are the points corresponding to $w=0$ and $w=\infty$, which are already excluded anyways.

{\it Step 2: Local behavior of specializations (one prime at a time)}: 
So far, we have worked on the geometric level, and basically (up to killing ramification at $0$ and $\infty$, which is easily achieved) obtained extensions fulfilling the assumptions of  Lemma \ref{lem:sqf-loccyc} over $\mathbb{Q}(s)$ instead of $\mathbb{Q}$. The tool to pull these results back to $\mathbb{Q}$ is Theorem \ref{thm:kln}, but we face the problem that there are always finitely many ``exceptional" primes. 
We will now make some additional considerations on residue extensions of the extension $F^w/\mathbb{Q}(s)$, in order to control the exceptional primes in Theorem \ref{thm:kln}. 
Firstly, the decomposition group at $t\mapsto 0$ of the splitting field of $g(t,X):=X^2(X-1)^3-t(X-a(w))$ over $\mathbb{Q}(t)$ is the split extension $D_6 = C_6\rtimes C_2$ of the inertia group $\langle(1,2)(3,4,5)\rangle \cong C_6$. The completion of the above splitting field over $t\mapsto 0$ therefore contains a totally ramified degree-$6$ subextension, and it is well known that such an extension is necessarily of the form $\mathbb{Q}((\sqrt[6]{bt}))/\mathbb{Q}((t))$ for some $b\in \mathbb{Q}^\times$, see, e.g., \cite[pp.52--53]{Lang}. Furthermore, since the residue extension at a branch point of ramification index $e$ always contains the $e$-th roots of unity, this residue extension must equal $\mathbb{Q}(\zeta_3)$. The completion at $t\mapsto 0$ therefore equals $\mathbb{Q}(\zeta_3)((\sqrt[6]{bt}))$ for some $b\in \mathbb{Q}^\times$.  
To determine the value of $b$, one first lets $X\to 0$ and then $Y:=X-1\to 0$, to see that the completion contains the fields $\mathbb{Q}((\sqrt{at}))$ as well as $\mathbb{Q}((\sqrt[3]{(1-a)t}))$, making the completion in total equal to $\mathbb{Q}(\zeta_3)((\sqrt[6]{\frac{a^3}{(1-a)^2}t}))$. 
%Updated:t=-x^2/-a, i.e., root(at), and t=y^3/(1-a)

Now let the transcendental $r$ fulfill $r^6=c(w)\cdot \gamma^2 t$ (with $\gamma\ne 0$ to be determined). 
The splitting field of $g(t(r),X) = X^2(X-1)^3-\frac{r^6}{c\gamma^2}(X-a) = X^2(X-1)^3 - c(\frac{r^3}{c\gamma})^2(X-a)$ over $\mathbb{Q}(r)$ is then clearly also a pullback of $F^w/\mathbb{Q}(s)$; 
it is furthermore unramified at $r\mapsto 0$ due to Abhyankar's lemma (since $\mathbb{Q}(r)/\mathbb{Q}(t)$ is totally ramified of ramification index $6$ at $t\mapsto 0$). Combining with Krasner's lemma (namely, specializing $s=s(r)\in \mathbb{Q}(r)$ at values $r\mapsto r_0$ sufficiently close to $0$ in the $p$-adic topology), we can then conclude that every completion (at some prime $p$) of $(F^w(r))_{r\mapsto 0}/\mathbb{Q}$ is also the completion of some specialization $F^w_{s\mapsto s_0}/\mathbb{Q}$ at some {\it non-branch-point} $s_0\in \mathbb{Q}$ of $F^w/\mathbb{Q}(s)$. To get hold of what this completion looks like, we next identify  $\textrm{Gal}((F^w(r))_{r\mapsto 0}/\mathbb{Q})$ as a subgroup of $\textrm{Gal}(F^w(r)/\mathbb{Q}(r))\cong S_5$. 
The completion at $t\mapsto 0$ of the whole Galois closure of $F^w(r)/\mathbb{Q}(t)$ is the compositum of $\mathbb{Q}(\zeta_3)((\sqrt[6]{\frac{a^3}{(1-a)^2}t})) = \mathbb{Q}(\zeta_3)((\sqrt[6]{\frac{a^3}{c\gamma^2(1-a)^2}}\cdot r))$ and $\mathbb{Q}((r))$, and hence equal to $\mathbb{Q}(\zeta_3, \sqrt[6]{\frac{a^3}{c\gamma^2(1-a)^2}})((r))$. Setting now $\gamma:=(1-a)^2c$ and noting that $2a = 4w^2$ is a square, we see that the residue extension of $F^w(r)/\mathbb{Q}(r)$ at (the non-branch point) $r\mapsto 0$ becomes the (at most) biquadratic extension $\mathbb{Q}(\zeta_3, \sqrt{\frac{c}{2}}) = \mathbb{Q}(\zeta_3, \sqrt{(-50w^2+27)(10w^2+8w+1)})$.
Note that any prime ramifying in this extension is either a divisor (of odd multiplicity) of $c/2$, or equal to one of $2$, $3$ or $\infty$. 
Redoing the above calculation individually for the completions of the (degree-$5$) extension $\mathbb{Q}(x)/\mathbb{Q}(t)$ at the two places extending $t\mapsto 0$ (namely, $x\mapsto 0$ and $x\mapsto 1$), one furthermore gets an embedding of the decomposition group $\textrm{Gal}((F^w(r))_{r\mapsto 0}/\mathbb{Q}) \cong C_2\times C_2$ into $S_5$ identifying $\textrm{Gal}(\mathbb{Q}(\sqrt{c/2})/\mathbb{Q})$ with $\langle(1,2)\rangle$ and $\textrm{Gal}(\mathbb{Q}(\zeta_3)/\mathbb{Q})$ with $\langle (3,4)\rangle$ (up to conjugation in $S_5$).
 
 Assume now, as we have done in condition a) of the theorem, that we can find $w$ such that 
all primes $p$ of odd multiplicity in $c/2$ are congruent to $1$ mod $3$. Then $p$ splits in $\mathbb{Q}(\zeta_3)$ and conversely $\sqrt{|c/2|}$ is in $\mathbb{Q}_3$. But note that condition b) on the values $w$ ensures that $c>0$, whence $3$ is indeed split in $\mathbb{Q}(\sqrt{c/2})$. 
%Instead can remove extra considerations on infinity!
 We have therefore obtained that all the finite ramified primes of $F^w(r)_{r\mapsto 0}/\mathbb{Q}$ dividing $c/2$, and also the prime $3$, have decomposition group generated by a transposition in this specialization.
%\marginpar{Need identification of subgroups of the decomposition group for this!! Treat two completions (at $x$ and $1-x$) in the degree-$5$ extension individually!}
%up to linear transformation, one gives x^3=r^6, and the other one y^2=(c/2)r^6

Moreover, there exist values $s_0$ such that $F^w_{s_0}/\mathbb{Q}$ is unramified at $2$; indeed, one verifies computationally (see the ancillary file) that, for $w_1$ odd and $w_2$ even (resp, $w_1$ even and $w_2$ odd), the discriminant of $F^w_{s\mapsto (w_2^3)/4}/\mathbb{Q}$ (resp., of $F^w_{s\mapsto 2w_2^3}$) is equal to the product of a square and an integer congruent to $1$ modulo $4$; finally, when $w_1$ and $w_2$ are both odd, the same holds for one of the two specializations $F^w_{s\mapsto 2w_2^3}/\mathbb{Q}$ and $F^w_{s\mapsto 6w_2^3}/\mathbb{Q}$, i.e., in all cases, we have identified a specialization $F^w_{s_0}/\mathbb{Q}$ in whose quadratic subextension the prime $2$ is unramified.
  
 On the other hand, when $w_2$ is odd and $s_0=2w_2^3$ (and same for $s_0=6w_2^3$),
 the defining polynomial $f_{w,s_0}(X)$ factors mod $2$ after suitable linear shifts as $\frac{1}{8} f_{w_1/w_2, 2w_2^3}(1-2X) = (2X-1)^2\cdot (-X)^3+(50w_1^2-27w_2^2)(10w_1^2+8w_1w_2+w_2^2)(w_2^2(1-2X)-2w_1^2) \equiv X^3+1$ mod $2$, which is a separable cubic. This means that over $\mathbb{Q}_2$ this polynomial has exactly three roots of non-negative valuation, all in the unramified extension. Hence the inertia group at $2$ in this specialization is at most cyclic of order $2$, generated by a transposition (and the same result follows analogously for the other specializations considered in the previous paragraph). But since we have also just seen that this inertia group must be contained in $A_5$, we conclude altogether that $2$ is unramified in this specialization. The case ``$w_2$ even" is easier, since here the specialization $f_{w_1/w_2, w_2^3/4}(X)$ is congruent, up to constant, to a separable quintic mod $2$.

Finally, the condition b') ensures the existence of values $s_0$ such that complex conjugation in $F^w_{s_0}/\mathbb{Q}$ is represented by the identity, i.e., the extension is totally real. To verify this claim, one need not try out infinitely many values $w$; rather, for a given $w$, the cycle type of conjugation is locally constant as $s_0$ moves through $\mathbb{P}^1(\mathbb{R})$ minus the set of branch points of $F^w$; and furthermore, it is computable in a purely group-theoretical way in terms of the ``branch cycle description" of the cover corresponding to $F^w/\mathbb{Q}(s)$ (see, e.g., \cite[Chapter I, Theorem 10.3]{MM} for a precise formula). This branch cycle description is again locally invariant and can change only when $w$ passes through a boundary point of the ``Hurwitz space" underlying our family of extensions, i.e., when $w$ is specialized such that two or more branch points collide. Via a calculation of discriminants, the problem thus reduces to evaluating a finite number of $w$ and, for each $w$-values, a finite number of $s$-values, counting the number of real roots of $f_{w,s}(X)$ for each choice.

%Since the problem corresponds to finding rational points on certain twists of the genus-$1$ curve $Y^2 = (43w^2-200w-50)(97w^2+64w+10)$, there should indeed be infinitely many $w$ fulfilling this; existence of {\it some} $w$ is easy to verify.

%Deal with bad primes 2 and infinity separately.
%maybe at 2, general solution when w is odd, similarly as for case b) below.
% X^3(X-1)^2-ct^2(9w^2X-(w^2-8w-2)) ---> discriminant over Z[w,t]; when w and t are odd [also: when numerator of w is odd, and t is odd!!], this becomes 1 mod 4 (up to squares), so 2 is unramified in quadratic!
% On the other hand, factorization pattern of f(2X) shows in.gp <= <(1,2)>, and so altogether: "unramified at 2"!

% For infinity: totally real only in rather small segments (e.g. vicinity of w=-1/4)...
% Argument for existence of ``transposition conjugation"? (Clear as soon as one of the transposition branch points is real)

%
% Sufficient conditions on w to allow desired specializations; (put this in theorem statement!!)

{\it Step 3: Local behavior of specializations (all primes simultaneously)}: 
We are now ready to invoke Theorem \ref{thm:kln}. Let $\mathcal{S}_0$ be the exceptional set (as in Theorem \ref{thm:kln}) of finite primes for the extension $F^w/\mathbb{Q}(s)$. We know already that, for each $p\in \mathcal{S}_0$, there exists a non-branch-point $s(p)\in\mathbb{Q}$ such that $F^w_{s(p)}/\mathbb{Q}$ is either unramified at $p$, or tamely ramified with cyclic decomposition group generated by a transposition.\footnote{In fact, we have seen that we can avoid ramification for all $p\in \mathcal{S}_0$, except possibly $3$, $\infty$ and the primes of odd multiplicity in $c/2$.} Combining with Krasner's lemma, we obtain the existence of a non-empty $\mathcal{S}_0$-adically open set (i.e., intersection of $p$-adic neighborhoods for $p\in \mathcal{S}_0$) of values $s_0$ such that $F^w_{s_0}/\mathbb{Q}$ fulfills the above for all primes $p\in \mathcal{S}_0$ simultaneously. We may even assume additionally that $s_0$ is of the form $s_0=u^3/v^2$ with $u,v\in \mathbb{Z}$ coprime, since the set of these numbers is easily seen to be dense in every $\mathbb{Q}_p$.\footnote{This last assertion is true also for $p=\infty$, whence, under the extra assumption of condition b'), we may also assume the specialization $F^w_{s_0}$ to be totally real up to suitable choice of $s_0$.} This extra assumption has the effect that, when applying Theorem \ref{thm:beck}, primes outside of $\mathcal{S}_0$ at which $s_0$ meets either the branch point $s\mapsto 0$ or the branch point $s\mapsto \infty$ will remain unramified in $F^w_{s_0}/\mathbb{Q}$ (since the ramification indices at these two branch points are $3$ and $2$ respectively). Ramified primes, other than primes in $\mathcal{S}_0$, will therefore all have inertia group generated by a transposition, corresponding to the inertia group at any finite non-zero branch point of $F^w/\mathbb{Q}(s)$. But at those branch points, by the construction in Step 1, the decomposition group is contained in $C_2\times C_3$; due to Theorem \ref{thm:kln}, the same will therefore hold for all ramified primes (outside of $\mathcal{S}_0$) of $F^w_{s_0}/\mathbb{Q}$.  Due to the well-known compatibility of Hilbert's irreducibility theorem with weak approximation, we may even assume $F^w_{s_0}/\mathbb{Q}$ to have full Galois group $S_5$ for infinitely many such $s_0$. It is also immediately obvious that there are infinitely many distinct quadratic subextensions inside these infinitely many specializations (for example, suitable choice of $s_0$ yields arbitrarily large primes with ``transposition inertia").

%Want this value to be of form s_0 = u^3/v^2?

{\it Step 4: Solving the embedding problem}:  
In total, we have obtained infinitely many (tamely ramified) specializations $F^w_{s_0}/\mathbb{Q}$ in which all non-trivial inertia groups are generated by transpositions, and all decomposition groups at such primes are cyclic (of order $2$ or $6$). 
%By Proposition \ref{lem:prime_kernel}, the embedding problem is (properly, since the extension is non-split) solvable, and by Proposition \ref{prop:serre}, the solution field $L/\mathbb{Q}$ may even be chosen such that $L/F^w_{s_0}$ is unramified at all finite primes.
By Lemma \ref{lem:sqf-loccyc}, each such specialization embeds into a quadratic extension $L\supseteq F^{w}_{s_0}$ such that, with $k:=\mathbb{Q}(\sqrt{\Delta(F^w_{s_0})})$, the extension $L/k$ is Galois of group $SL_2(5)$ and unramified at all finite primes.
 When $k$ is imaginary, we are already done. Otherwise, we may need to change $L$ slightly to make $L/k$ unramified also at the infinite primes. Note that in this case, complex conjugation in $F^w_{s_0}/k$ would have to be represented by an element of $A_5$, i.e., by the identity (the case of a double transposition as complex conjugation does not occur, since this would be a local obstruction to the embedding problem, which we have already excluded).

% actually success regarding exceptional set comes with specialization s_0=c*u^3/v^2...
% specifically for prime 43, this should mean "transposition ramification" could be recognized in the original (not translated) extension from t=0!
% (and here, dec.gp. is in C_6 as soon as prime is (not in exceptional set and) =1 mod 3.
%

So we are left with the case that $F^w_{s_0}$ is totally real, but $L$ is not any more. In this case, we can twist $L/k$ to obtain a totally real (i.e., unramified at the archimedean prime) extension without changing ramification at other primes. 
%\marginpar{Now, need to adapt to general situation!!}
%should take k(\sqrt{-p})/k, for any prime divisor p=3 mod 4 of the discriminant; but does that always exist??
% Only know that disc(k) is 1 mod 4 %%% Need extra condition?
%discriminant in total is -(c/2)*p(w,t), where p(w1,w2,t) mod 4 equals w1^12*t^2; also c/2 has two factors, one of which is  = 3 mod 8! ---> at least one prime divisor 3 mod 4, at least as soon as this factor is positive?!?!
%
% For w in relevant interval, both factors of c/2 are positive! (which means p(w1,w2,t) must be negative)
% But also, last factor is negative, so its absolute value is 3 mod 4! Could there be cancellation?
%
Recall that, from our considerations in Step 2, we may assume that $s_0$ is chosen such that all primes dividing $(-50w_1^2+27w_2^2)(10w_1^2+8w_1w_2+w_2^2)$ an odd number of times have transposition inertia in $F^w_{s_0}/\mathbb{Q}$, and are therefore in particular ramified in $k/\mathbb{Q}$. 
But note that, since at least one of $w_1, w_2$ is odd, the above product always evaluates either to an integer congruent to $3$ mod $4$ or to four times such an integer.
Also, in the intervals allowing totally real specializations, both quadratic factors take only positive values. This means that there is necessarily at least one prime $p\equiv 3$ mod $4$ ramifying in $k/\mathbb{Q}$. We may then take take the compositum $\widehat{L}$ of $L$ and $k(\sqrt{-p})$. This is an $SL_2(5)\times C_2$-extension of $k$, and if $D\cong C_2$ denotes the diagonal subgroup of the center of the group $SL_2(5)\times C_2$, then $\tilde{L}^D/k$ is again an $SL_2(5)$-extension. Since $k/\mathbb{Q}$ is already ramified at $p$, the extension $k(\sqrt{-p})/k$ is ramified exactly at the archimedean primes. Hence, all non-archimedean primes still remain unramified in $\tilde{L}^{D}/k$, whereas for the archimedean ones,\footnote{Note that both archimedean primes of $k$ must behave identically in $L$, since $L/\mathbb{Q}$ is Galois.} the fact that $L/k$ and $k(\sqrt{-p})/k$ have the same completion (namely $\mathbb{C}$) means that $\tilde{L}^{D}\subset \mathbb{R}$. This completes the proof.
\end{proof}
 
Note once again that, using the above proof, Theorem \ref{thm:q5} can be made fully effective in the sense that, for each $w$ fulfilling the assumptions, an explicit infinite set of values $s_0$ (and hence, via calculation of discriminants, an explicit set of quadratic fields fulfilling the assertion) can be computed; indeed, as we have seen, all $s_0 =\frac{u^3}{v^2}$ with $u,v\in \mathbb{Z}$ in some suitable intersection of finitely many $p$-adic neighborhoods are admissible. Making these neighborhoods explicit is also possible, although the result might not look too attractive.

For the proof of the explicit assertion of Corollary \ref{cor:q5}, it suffices to investigate the local behavior of the specializations $F^w_{s_0}/\mathbb{Q}$ (with $w=-2/3$) at the ``exceptional" primes % in $\mathcal{S}_0$
 somewhat more closely.
 \begin{proof}[Proof of Corollary \ref{cor:q5}]
 Setting $w=-2/3$ in the notation of the proof of Theorem \ref{thm:q5}, we have $a(w)=8/9$ and $c(w)=86/81$. Up to multiplying $s$ with a suitable constant, we thus have the defining polynomial $f(s,X) = X^2(X-1)^3 - 86s^2(9X-8)$ for the extension $F^w/\mathbb{Q}(s)$.
 Using Remark \ref{rem:explicit}, a calculation of discriminants shows that the set of exceptional primes (in the sense of Theorem \ref{thm:kln}) is here $\mathcal{S}_0\cup \{\infty\}$ with $\mathcal{S}_0 = \{2,3,5,43,97\}$.
 We want to specialize $s\mapsto s_0 = 86 u^3$ with $u\in \mathbb{Z}$ coprime to $30$, which (again, by a discriminant calculation) gives exactly the quadratic subfields as in the assertion of the corollary. Just as in the proof of Theorem \ref{thm:q5}, this specialization ``kills" ramification inherited from branch points with non-transposition inertia, and we conclude that $F^w_{s_0}/\mathbb{Q}$ will have no local obstruction for the central embedding problem at finite primes outside of $\mathcal{S}_0$. We are therefore reduced to an investigation of local behavior at those finitely many exceptional primes. 
 %At infinity, there will be no obstruction, since the quadratic subfield is in fact always imaginary here, meaning that the decomposition group at infinity is generated by a transposition in $S_5$. 
 
 %Obvious from factorization pattern for 5.

%for 2: either separable, or (after shift) separable degree 3 factor; but what if in.gp. <(1,2)>?
% unramifiedness clear from shape \mathbb{Q}(\sqrt{-43\cdot (2^{11}\cdot 3^{10}\cdot 43^6 u^{12} - 43^3\cdot 263\cdot 883\cdot u^6 + 108)}) ?
% Indeed, since expression under root is 1 mod 4!!!!

% for 97: seems either separable or reduction mod 97 (after shift) separable cubic with three roots --> at worst transposition dec.gp.

%whereas for the prime $43$, we already verified above that no obstruction occurs for the values $s_0$ in question. 
%\marginpar{Rewrite!}
Regarding the prime $43$, note that if $s_0=86(43^{d} u')^3$ ($u'\in \mathbb{Z}$ coprime to $43$, and $d\ge 0$), then the mod-$43$ reduction of the polynomial $\frac{1}{43^{6d+3}} f(s_0, 1-43^{2d+1}X)$ 
 is the separable cubic polynomial $-X^3+-2^3u'^6$ with three roots. This means that ${\textrm{Gal}}(f(s_0,X)/\mathbb{Q}_{43})$ is either trivial or generated by a transposition (in fact, since $43$ divides the discriminant of $F^w_{s_0}$, the latter is the case), meaning that for these values $s_0$, there is no local obstruction at $43$. The same reasoning works for the primes $2$, $3$, $5$ and $97$ (in fact, one may even additionally deduce that these primes are always unramified under the given conditions, as the explicit expression for the discriminant of the quadratic subfield shows).

It thus remains to verify that $\textrm{Gal}(F^w_{s_0}/\mathbb{Q})$ is always all of $S_5$. With our specifications, $f(s_0,X)$ is always separable modulo $5$, and either irreducible or divisible by an irreducible degree-$4$ factor; similarly, $f$ is always separable modulo $3$, with irreducible factors of degree $2$ and $3$. 

Now $\textrm{Gal}(F^w_{s_0}/\mathbb{Q})=S_5$ follows from standard results, since a subgroup of $S_5$ containing an element of cycle structure $(3.2)$ and one of cycle structure either $(5)$ or $(4.1)$ is necessarily all of $S_5$.
\end{proof}

\appendix
\section{Remarks on the case $G=SL_2(9)$}
We conclude by noting that the above approach could in principle work analogously to solve Conjecture \ref{conj:1} also for the group $G=SL_2(9)$ In particular, due to the isomorphism $PSL_2(9) \cong A_6$, one may proceed in analogy to the previous section, by first producing Galois extensions of $\mathbb{Q}$ with Galois group $S_6$ and suitable local behavior, and then embedding them into extensions with group $2.S_6$, the degree-$2$ central extension of $S_6$ in which the transpositions are split. The latter group has an index-$2$ normal subgroup isomorphic to $SL_2(9)$. We thus have a sufficient criterion for the solvability of the embedding problem similar to Lemma \ref{lem:sqf-loccyc}:
\begin{lemma}
Let $\Omega/\mathbb{Q}$ be a Galois extension with group $S_6$ such that
\begin{itemize}
\item[1)] all inertia groups at ramified primes of $\Omega/\mathbb{Q}$ are generated by transpositions, and
\item[2)] all decomposition groups at ramified primes are conjugate in $S_6$ to one of $\langle(1,2)\rangle$, $\langle(1,2), (3,4)(5,6)\rangle$ and $\langle(1,2), (3,4,5)\rangle$.
\end{itemize}
Then, if $k\subset \Omega$ denotes the fixed field of $A_6$, the extension $\Omega/k$ embeds into an $SL_2(9)$-extension unramified at all finite primes.
\end{lemma}
\begin{proof}
It suffices to note that the first and third group in condition 2) are split in the extension $2.S_6$ (and thus pose no local obstruction to the central embedding problem), whereas the second one lifts to an abelian group isomorphic to $C_2\times C_4$ (with the subgroup generated by the transposition still split); in case a prime $p$ has this decomposition group (and transposition inertia) in $\Omega/\mathbb{Q}$, there is no local obstruction to the embedding problem either, since the local extension is then isomorphic to the compositum of a totally ramified quadratic extension with the unramified quadratic extension, and the local embedding problem is solved by extending the quadratic unramified extension to the degree-$4$ one. The rest of the proof is simply application of Propositions \ref{lem:prime_kernel} and \ref{prop:serre}, as for Lemma \ref{lem:sqf-loccyc}. 
\end{proof}

In order to enforce conditions 1) and 2) of the above lemma, one may again construct geometric Galois realizations with inertia groups generated by transpositions, and this time such that the decomposition group at such branch points is contained in $\textrm{Sym}(\{1,2\}) \times \textrm{Alt}(\{3,4,5,6\})$; this is again a subgroup of index $2$ in the largest possible decomposition group $S_2\times S_4$ for a ``transposition branch point", and hence enforcing such a situation corresponds to a process of ``swallowing" discriminants just as in the previous section, which may again be performed using certain $4$-branch point families. So far, however, I have not succeeded in additionally controlling the ``small" (i.e., exceptional) primes for such families, and it may well be necessary to increase the branch point number further to achieve this.

\end{document}